\documentclass[12pt,a4paper,oneside,reqno]{article}

\usepackage[utf8]{inputenc}
\usepackage[english]{babel}
\usepackage[left=1in,right=1in,top=1in,bottom=1in]{geometry}
\usepackage{amsfonts,amssymb,amsmath,amsthm}
\usepackage[symbol]{footmisc}
\usepackage{longtable,array}
\newcolumntype{L}[1]{>{\raggedright\arraybackslash$}p{#1}<{$}}
\usepackage{enumerate}
\usepackage{cite}
\usepackage{xcolor}

\usepackage{hyperref}
\hypersetup{hidelinks}
\usepackage{fancyhdr}
\theoremstyle{plain}
\newtheorem{theorem}{Theorem}
\newtheorem{lemma}[theorem]{Lemma}
\newtheorem{proposition}[theorem]{Proposition}
\newtheorem{corollary}[theorem]{Corollary}
\theoremstyle{definition}
\newtheorem{remark}[theorem]{Remark}
\newtheorem{definition}[theorem]{Definition}

\makeatletter
\@removefromreset{theorem}{section}\@removefromreset{equation}{section}

\makeatother
\usepackage{adjustbox}
\usepackage{tikz}\usetikzlibrary{arrows}

\IfFileExists{fouriernc.sty}{\usepackage{fouriernc}}{}

\newcommand{\doi}[1]{\href{https://doi.org/#1}{doi:\detokenize{#1}}}
\newcommand{\arxiv}[1]{\href{https://arxiv.org/abs/#1}{arXiv:#1}}
\newcommand{\arxivsearch}[1]{\href{https://arxiv.org/search/?query=#1}{arXiv}}

\newtheorem{thrm}[theorem]{Theorem}

\newtheorem{lem}[theorem]{Lemma}

\newtheorem{defn}[theorem]{Definition}

\newcommand{\gen}[1]{\langle #1\rangle}

\newcommand{\LL}{\mathfrak{L}}

\providecommand{\Ann}{\operatorname{Ann}}

\begin{document}

\noindent\fbox{\parbox{\dimexpr\textwidth-2\fboxsep-2\fboxrule\relax}{\small
\textbf{Corrected version (verified by Claude).} This is a corrected version of the paper I. Kaygorodov, A. Khudoyberdiyev,
\emph{Transposed Poisson structures on solvable and perfect Lie algebras},
Journal of Physics A: Mathematical and Theoretical, 57 (2024), no.~3, 035205,
\doi{10.1088/1751-8121/ad1620}.}}

\bigskip

\begin{center}
{\Large Transposed Poisson structures on\\ solvable and perfect Lie algebras}

\bigskip

{\bf Ivan Kaygorodov and Abror Khudoyberdiyev}
\end{center}

\medskip

\noindent{\small\textbf{Abstract.} We described all transposed Poisson algebra structures on 
oscillator Lie algebras, i.e., on   one-dimensional 
solvable extensions of the $(2n+1)$-dimensional Heisenberg algebra; 
 on solvable Lie algebras with naturally graded filiform
nilpotent radical;
 on $(n+1)$-dimensional solvable   extensions of the $(2n+1)$-dimensional Heisenberg algebra; 
 and on $n$-dimensional solvable   extensions of the $n$-dimensional  algebra with trivial multiplication.
We also answered one question on transposed Poisson algebras early posted   in 
a paper by Beites, Ferreira and Kaygorodov. 
Namely, we found 
that the semidirect product of ${\mathfrak sl}_2$ and irreducible module gives 
a finite-dimensional Lie algebra with non-trivial $\frac{1}{2}$-derivations, 
but without non-trivial transposed Poisson structures.}

\bigskip

\noindent {\bf Keywords:}
{\it Lie algebra, transposed Poisson algebra, $\delta$-derivation, oscillator Lie algebra, solvable Lie algebra.}

\bigskip

\noindent {\bf MSC2020:} 17B63 (primary); 17A30, 17B30, 17B40, 17B61 (secondary).

\bigskip

\tableofcontents

\bigskip

\section*{Introduction}\addcontentsline{toc}{section}{Introduction}

	Since their origin in the 1970s in Poisson geometry, Poisson algebras have appeared in several areas of mathematics and physics, such as algebraic geometry, operads, quantization theory, quantum groups, and classical and quantum mechanics. One of the natural tasks in the theory of Poisson algebras is the description of all such algebras with a fixed Lie or associative part~\cite{YYZ07,said2,jawo,kk21}.
	Recently, Bai, Bai, Guo,  and Wu have introduced a dual notion of the Poisson algebra~\cite{bai20}, called a  transposed Poisson algebra, by exchanging the roles of the two multiplications in the Leibniz rule defining a Poisson algebra. A transposed Poisson algebra defined this way not only shares some properties of a Poisson algebra, such as the closedness under tensor products and the Koszul self-duality as an operad but also admits a rich class of identities \cite{kms,bai20,fer23,bfk22,lb23,bl23}. It is important to note that a transposed Poisson algebra naturally arises from a Novikov-Poisson algebra by taking the commutator Lie algebra of its Novikov part \cite{bai20}.
 	Any unital transposed Poisson algebra is
	a particular case of a ``contact bracket'' algebra 
	and a quasi-Poisson algebra \cite{bfk22}.
 Each transposed Poisson algebra is a 
 commutative  Gelfand-Dorfman algebra \cite{kms}
 and it is also an algebra of Jordan brackets \cite{fer23}.
Transposed Poisson algebras are related to weak Leibniz algebras \cite{dzhuma}.
	In a recent paper by Ferreira, Kaygorodov, Lopatkin
	a relation between $\frac{1}{2}$-derivations of Lie algebras and 
	transposed Poisson algebras has been established \cite{FKL}. 	These ideas were used to describe all transposed Poisson structures 
	on  Witt and Virasoro algebras in  \cite{FKL};
	on   twisted Heisenberg-Virasoro,   Schr\"odinger-Virasoro  and  
	extended Schr\"odinger-Virasoro algebras in \cite{yh21};
	on Schr\"odinger algebra in $(n+1)$-dimensional space-time in \cite{ytk};
on solvable Lie algebra with filiform nilradical in \cite{aae23};
	on Witt type Lie algebras in \cite{kk23};
	on generalized Witt algebras in \cite{kkg23}; 
 Block Lie algebras in \cite{kk22,kkg23};
 on the Lie algebra of upper triangular matrices in \cite{KK7};
  and
on    Lie incidence algebras in \cite{kkinc}.
It was proved that any complex finite-dimensional solvable Lie algebra admits a non-trivial transposed Poisson structure \cite{klv22}. The algebraic and geometric classification of $3$-dimensional transposed Poisson algebras was given in \cite{bfk23}. For the list of current open questions on transposed Poisson algebras, see \cite{bfk22}.

 \medskip

From the point of view of physical (and other) applications, we note that solvable Lie algebras often occur as Lie algebras of symmetry groups of differential equations.
Algorithms exist for determining these symmetry algebras \cite{4,5}. 
Subalgebras of the symmetry algebra can then be used to perform symmetry reduction, i.e., to construct group-invariant solutions. An important step in this procedure is to identify the symmetry algebra
(and its subalgebras) as abstract Lie algebras. Again, algorithms exist for realizing such identification in broad terms (like decomposing $L$ into a direct sum or obtaining its Levi
decomposition, etc.) \cite{6}. 
 The Heisenberg algebra $H(n)$ with the basis 
 $\{P_1, \ldots, P_n, B_1, \ldots, B_n, H \}$ is also a subalgebra of the quantum mechanical Galilei algebra   (or extended Galilei algebra) \cite{11,12}. The operators $P_i$ in this interpretation generate space translations, $B_i$ generate Galilei boost, and $H$ changes the phase of the wave function. The algebra $H(n)$ is also a subalgebra of the symmetry algebra of the heat equation \cite{4}, and of the nonlinear Schr\"odinger equation with any nonlinearity $F(|\psi|)$, depending only on the absolute values of the wave function $\psi$ and of many other partial differential equations occurring in non-relativistic physical theories \cite{14,15}. In this context, extensions of the Heisenberg algebra $H(n)$ are part of a study of physical theories with symmetries going beyond translations and Galilei boosts.
 The class of oscillator Lie algebras, which are the only noncommutative solvable Lie algebras that carry a bi-invariant Lorentzian metric \cite{medina85}, is a class of quadratic Lie algebras. The former algebras, which are the Lie algebras of oscillator Lie groups,
have aroused some interest \cite{biggs,said2,cal,crampe,bm11,angel}.
Ballesteros and  Herranz obtained all possible coboundary Lie bialgebras for the four-dimensional oscillator algebra \cite{angel}.
Biggs and  Remsing 
  examined the algebraic structure of the four-dimensional oscillator Lie algebra\cite{biggs}. Firstly, the adjoint orbits are determined (and graphed); these orbits turn out to be linearly isomorphic to the coadjoint orbits. This prompts the identification of a family of invariant scalar products (from which the Casimir functions can be recovered). They identified the group of automorphisms (resp. inner automorphisms); subsequently, they classified all linear subspaces. As a by-product, they arrived at classifications of the full-rank subspaces, the subalgebras, and the ideals.
 Based on the four-dimensional oscillator Lie algebra,
 Crampé, van de Vijver, and  Vinet introduced an associative algebra called the oscillator Racah algebras and explained a relation between the representation theory of and multivariate Krawtchouk polynomials \cite{crampe}.
  In \cite{cal}, Calvaruso and Zaeim obtained a classification of all Ricci, curvature, Weyl, and matter collineations for the four-dimensional oscillator group equipped with a 1-parameter family of left-invariant Lorentzian metrics.
  Boucetta and Medina 
  determined the Lie bialgebra structures and the solutions of the classical Yang–Baxter equation on a generic class of oscillator Lie algebras \cite{bm11}.
   Albuquerque, Barreiro,  Benayadi,   Boucetta, and   Sánchez described all Poisson algebra structures and symmetric Leibniz bialgebra structures on generic oscillator Lie algebras \cite{said2}.
 The present paper studies transposed Poisson algebra structures on oscillator Lie algebras.
Namely, we have new examples of transposed Poisson structures on oscillator Lie algebras.

\medskip

The structure of the present paper is as follows. 
In Section \ref{prem},  we discuss some definitions and previous results. 
In Section \ref{tpaoscl}, we give the full description of all non-isomorphic transposed Poisson structures on generic oscillator Lie algebras, i.e., on   one-dimensional 
solvable extensions of the $(2n+1)$-dimensional Heisenberg algebra.
Section \ref{solv} is dedicated to a study of transposed Poisson algebra structures on some solvable algebras with a special type of nilpotent radical, obtained in some papers by Winternitz with co-authors \cite{sw05,nw94,rw93}.
Namely, we described all transposed Poisson algebra structures 
on   solvable Lie algebras with naturally graded filiform
nilpotent radical;
 on $(n+1)$-dimensional solvable   extensions of the $(2n+1)$-dimensional Heisenberg algebra; 
 and on $n$-dimensional solvable   extensions of the $n$-dimensional  algebra with the trivial multiplication.
In the last section, we also answered one question on transposed Poisson algebras posted early in a paper by Beites, Ferreira, and Kaygorodov. Namely, we found a finite-dimensional Lie algebra with non-trivial $\frac{1}{2}$-derivations, without non-trivial transposed Poisson algebra structures.

 \section{Preliminaries}\label{prem}
	All the algebras below will be over the complex field, and all the linear maps will be $\mathbb C$-linear unless otherwise stated. The notation $\gen{S}$ means the $\mathbb C$-subspace generated by $S$.

	\begin{defn}\label{tpa}
		Let ${\mathfrak L}$ be a vector space equipped with two nonzero bilinear operations $\cdot$ and $[-,-].$
		The triple $({\mathfrak L},\cdot,[-,-])$ is called a \textit{transposed Poisson algebra} if $({\mathfrak L},\cdot)$ is a commutative associative algebra and
		$({\mathfrak L},[-,-])$ is a Lie algebra that satisfies the following compatibility condition
		\begin{center}
			$2z\cdot [x,y]=[z\cdot x,y]+[x,z\cdot y].$
		\end{center}
	\end{defn}
	
	Transposed Poisson algebras were first introduced by Bai, Bai, Guo, and Wu in \cite{bai20}.
	
	\begin{defn}\label{tp-structures}
		Let $({\mathfrak L},[-,-])$ be a Lie algebra. A \textit{transposed Poisson algebra structure} on $({\mathfrak L},[-,-])$ is a commutative associative multiplication $\cdot$ on $\mathfrak L$ which makes $({\mathfrak L},\cdot,[-,-])$ a transposed Poisson algebra.
	\end{defn}

	\begin{defn}\label{12der}
		Let $({\mathfrak L}, [-,-])$ be an algebra, $\varphi$ be a linear map
and $\phi$ be a bilinear map.
Then $\varphi$ is a $\frac{1}{2}$-derivation if it satisfies
\begin{center}
$\varphi[x,y]= \frac{1}{2} \big([\varphi(x),y]+ [x, \varphi(y)] \big);$
\end{center}
 $\phi$ is a $\frac{1}{2}$-biderivation if it satisfies
\begin{longtable}{crl}
$\phi([x,y],z)$&$=$&$ \frac{1}{2} \big( 
 [\phi(x,z),y] +[x,\phi(y,z)]\big),$\\ 
$\phi(x,[y,z])$&$=$&$ \frac{1}{2} \big([\phi(x,y),z]+ 
[y,\phi(x,z)] \big).$
\end{longtable}
 
	\end{defn}
	Observe that $\frac{1}{2}$-derivations are a particular case of $\delta$-derivations introduced by Filippov (see, for example,  \cite{fil1} and references therein).
 It is easy to see from Definition \ref{12der} that $[\LL,\LL]$ and $\Ann(\LL)$ are invariant under any $\frac 12$-derivation of $\LL$.

	Definitions \ref{tpa} and \ref{12der} immediately imply the following key Lemma.
	\begin{lem}\label{glavlem}
		Let $({\mathfrak L},[-,-])$ be a Lie algebra and $\cdot$ a new binary (bilinear) operation on ${\mathfrak L}$. Then $({\mathfrak L},\cdot,[-,-])$ is a transposed Poisson algebra 
		if and only if $\cdot$ is commutative and associative and for every $z\in{\mathfrak L}$ the multiplication by $z$ in $({\mathfrak L},\cdot)$ is a $\frac{1}{2}$-derivation of $({\mathfrak L}, [-,-]).$
	\end{lem}
	
	The basic example of a $\frac{1}{2}$-derivation is the multiplication by a field element.
	Such $\frac{1}{2}$-derivations will be called trivial. 
	
	\begin{thrm}[see \cite{FKL}]\label{princth}
		Let ${\mathfrak L}$ be a Lie algebra without non-trivial $\frac{1}{2}$-derivations.
		Then all transposed Poisson algebra structures on ${\mathfrak L}$ are trivial.
	\end{thrm}

	 Let $\cdot$ be a transposed Poisson algebra structure on a Lie algebra $({\mathfrak L}, [-,-])$. 
	 Then any automorphism $\phi$ of $({\mathfrak L}, [-,-])$ induces the transposed Poisson algebra structure $*$ on $({\mathfrak L}, [-,-])$ given by
	 \begin{align*}
		 x*y=\phi\big(\phi^{-1}(x)\cdot\phi^{-1}(y)\big).
		 \end{align*}
	 Clearly, $\phi$ is an isomorphism of transposed Poisson algebras $({\mathfrak L},\cdot,[-,-])$ and $({\mathfrak L},*,[-,-])$.

\section{Transposed Poisson algebra structures on oscillator Lie algebras }\label{tpaoscl}

\begin{definition}
The oscillator Lie algebra $\mathfrak{L_{\lambda}}$ is a $(2n+2)$-dimensional Lie algebra with its canonical basis $\mathbb{B} = \{ e_{-1}, e_{0}, e_j, \check{e}_j \}_{j=1,\ldots,n}$  
and the following multiplication rules:
 $$[e_{-1},e_j] = \lambda_j\check{e}_j, \ \ \  [e_{-1},\check{e}_j] = -\lambda_je_j, \ \ \ [e_j,\check{e}_j] = e_0,$$
for $j = 1, \ldots, n$ and $\lambda = (\lambda_1, \ldots, \lambda_n) \in \mathbb{R}^{n}$ with $0 < \lambda_1 \leq \ldots \leq \lambda_n$.
\end{definition}

\begin{proposition}\label{dmatrix} 
A linear map $\varphi : \mathfrak{L}_{\lambda} \rightarrow \mathfrak{L}_{\lambda}$ is a $\frac{1}{2}$-derivation of the algebra $(\mathfrak{L}_{\lambda}, [-,-])$ if and only if 
 
\begin{longtable}{lcl}
$\varphi(e_{-1})$ &$ =$&$ \gamma e_{-1} + \mu e_0 - \sum\limits_{j=1}^{n} 2\lambda_j \alpha_{j} e_j - \sum\limits_{j=1}^{n} 2 \lambda_j \beta_{j} \check{e}_j$,\\
$\varphi(e_0)$ &$ =$& $ \gamma e_0$,\\
$\varphi(e_j)$ & $=$ & $\alpha_{j} e_0 + \gamma e_j, \ \ \ (j=1, \ldots, n)$,\\
$\varphi(\check{e}_j)$ & $ =$ &$ \beta_{j} e_0 + \gamma \check{e}_{j}, \ \ \ (j=1, \ldots, n).$
\end{longtable}
\end{proposition}
\begin{proof}
 
For each $\ell \in \mathbb{B},$  we consider 
\begin{center}
$\varphi(\ell) = \varphi(\ell)_{-1}e_{-1} + \varphi(\ell)_{0}e_0 + \sum\limits_{k=1}^{n} (\varphi(\ell)_{k}e_k +  \varphi(\ell)_{\check{k}}\check{e}_k)$.
\end{center}

Since $[\LL,\LL]$ and $\Ann(\LL)$ are invariant under any $\frac 12$-derivation, we obtain that $\varphi(e_0) = \gamma e_0$ and $\varphi(e_j)_{-1} = \varphi(\check{e}_j)_{-1} =0.$

From $e_0 = [e_j,\check{e}_j]$ for each $j \in \{1, \ldots, n\},$ we get
$$\varphi(e_0) = \varphi([e_j,\check{e}_j]) = \frac{1}{2}([\varphi(e_j),\check{e}_j] + [e_j,\varphi(\check{e}_j)]) =
\frac{1}{2}(\varphi(e_j)_{j} + \varphi(\check{e}_j)_{\check{j}}e_0),$$
which implies  $\varphi(e_j)_j+\varphi(\check{e}_j)_{\check{j}} = 2\gamma.$
Since $[e_{-1},e_j] = \lambda_j \check{e}_j$ and $[e_{-1},\check{e}_j] = -\lambda_je_j,$ we obtain
\begin{longtable}{rcl}
$\varphi(\lambda_j\check{e}_j)$ &$ =$&$ \varphi([e_{-1},e_j]) = \frac{1}{2}([\varphi(e_{-1}),e_j] + [e_{-1},\varphi(e_j)])$ \\
&$=$& $\frac{1}{2} \lambda_j \varphi(e_{-1})_{-1} \check{e}_j - \frac{1}{2} \varphi(e_{-1})_{\check{j}} e_0   +\frac{1}{2} \sum\limits_{k=1}^n\Big(\lambda_k \varphi(e_j)_{k} \check{e}_k -\lambda_k \varphi(e_j)_{\check{k}} e_k\Big),$
\end{longtable}
\noindent and
\begin{longtable}{rcl}
$-\varphi(\lambda_je_j)$&$ =$&$ \varphi([e_{-1},\check{e}_j]) = \frac{1}{2}([\varphi(e_{-1}),\check{e}_j] + [e_{-1},\varphi(\check{e}_j)])$ \\
&$=$&$ -\frac{1}{2} \lambda_j \varphi(e_{-1})_{-1} e_j + \frac{1}{2} \varphi(e_{-1})_{j} e_0 +\frac{1}{2} \Big(\sum\limits_{k=1}^n \lambda_k \varphi(\check{e}_j)_{k} \check{e}_k -\lambda_k \varphi(\check{e}_j)_{\check{k}} e_k\Big).$
\end{longtable}

Hence,  for any $j \in \{1, \ldots, n\},$ we have 
$$\varphi(e_{-1})_{-1} = \varphi(e_j)_j= \varphi(\check{e}_j)_{\check{j}} =\gamma, \quad \varphi(e_{-1})_j=-2\lambda_j \varphi(e_j)_0, \quad  \varphi(e_{-1})_{\check{j}}=-2\lambda_j \varphi(e_{\check{j}})_0,$$ 
and
\begin{equation}\label{eq1}
 2\lambda_j \varphi(\check{e}_j)_k   = - \lambda_k \varphi(e_j)_{\check{k}}, \quad 2\lambda_j  \varphi(e_j)_{\check{k}}  = - \lambda_k \varphi(\check{e}_j)_k, \quad 1 \leq j, k \leq n,
\end{equation}

\begin{equation}\label{eq2}
 2\lambda_j \varphi(\check{e}_j)_{\check{k}}   =  \lambda_k \varphi(e_j)_k, \quad 2\lambda_j  \varphi(e_j)_k  =  \lambda_k \varphi(\check{e}_j)_{\check{k}}, \quad 1 \leq j, k (j \neq k) \leq n.
\end{equation}

Finally, from
\begin{longtable}{ccccc}
$0$& $=$ & $\varphi([e_j, e_k])$ &$=$ & $\frac{1}{2}([\varphi(e_j),e_k] + [e_j,\varphi(e_k)]),$\\
$0$ &$=$ & $\varphi([e_j, \check{e}_k])$ &$=$ & $\frac{1}{2}([\varphi(e_j),\check{e}_k] + [e_j,\varphi(\check{e}_k)]),$\\
$0$ &$=$ & $\varphi([\check{e}_j, \check{e}_k])$ &$=$ & $\frac{1}{2}([\varphi(\check{e}_j),\check{e}_k] + [\check{e}_j,\varphi(\check{e}_k)]),$
\end{longtable}
for $1 \leq j, k (j \neq k) \leq n,$ we derive that 
\begin{equation}\label{eq3}
 \varphi(e_j)_{\check{k}}   = \varphi(e_k)_{\check{j}}, \quad 
 \varphi(e_j)_k   = - \varphi(\check{e}_k)_{\check{j}}, \quad 
 \varphi(\check{e}_j)_k   = \varphi(\check{e}_k)_j, 
\end{equation}

It is not difficult to see that  from \eqref{eq1}, \eqref{eq2} and \eqref{eq3}, 
we get
 \begin{center} $ \varphi(e_j)_t=0$ for $t\in \{1,  \ldots, n, \check{1}, \ldots, \check{n}\} \setminus \{ j\},$\end{center}
\begin{center} $ \varphi(e_{\check{j}})_t=0$ for $t \in \{1,  \ldots, n, \check{1}, \ldots, \check{n}\} \setminus \{ \check{j}\}.$ \end{center}

Thus, we get that the map $\varphi$ has the following form:
\begin{longtable}{rcl}
$\varphi(e_{-1})$ &$ =$&$ \gamma e_{-1} + \mu e_0 - \sum\limits_{j=1}^{n} 2\lambda_j \alpha_{j} e_j - \sum\limits_{j=1}^{n} 2 \lambda_j \beta_{j} \check{e}_j$,\\
$\varphi(e_0)$ &$ =$& $ \gamma e_0$,\\
$\varphi(e_j)$ & $=$ & $\alpha_{j} e_0 + \gamma e_j, \ \ \ (j=1, \ldots, n)$,\\
$\varphi(\check{e}_j)$ & $ =$ &$ \beta_{j} e_0 + \gamma \check{e}_{j}, \ \ \ (j=1, \ldots, n).$
\end{longtable}
It is easy to check that all linear maps given by the present way are $\frac{1}{2}$-derivations.
It finishes the proof.

\end{proof}

 Filippov proved that each $\delta$-derivation ($\delta\neq0,1$) gives a non-trivial ${\rm Hom}$-Lie algebra structure \cite[Theorem 1]{fil1}.
 Hence, by Proposition \ref{dmatrix}, we have the following corollary.

  \begin{corollary}
  Each oscillator Lie algebra admits a non-trivial ${\rm Hom}$-Lie algebra structure.
 \end{corollary}

\begin{thrm}\label{theodot}
Let $(\mathfrak{L}_{\lambda},  [-,-])$ be the oscillator Lie algebra. Then every symmetric $\frac{1}{2}$-biderivation gives a transposed Poisson structure.
If $(\mathfrak{L}_{\lambda}, \cdot, [-,-])$ is a transposed  Poisson algebra, then  $(\mathfrak{L}_{\lambda}, \cdot)$ has  the following  multiplication:
\begin{longtable}{lcl}
$e_{-1} \cdot e_{-1}$ & $ =$ &$ \gamma  e_{-1}+ \mu e_0 - 2\sum\limits_{k=1}^{n} \lambda_k (\alpha_{k} e_k +\beta_{k} \check{e}_k),$\\
$ e_{-1} \cdot e_{0}$&$ =$&$ \gamma e_0,$\\
$e_{-1} \cdot e_{j}$ & $ =$ & $ \alpha_{j} e_0 + \gamma e_{j}, $\\
$e_{-1} \cdot \check{e}_{j}$ & $ =$ & $ \beta_{j} e_0 + \gamma \check{e}_{j},$\\
$e_j \cdot e_j$ & $ =$ & $ \check{e}_j \cdot \check{e}_j = -\frac{\gamma}{2\lambda_j} e_0  .$
\end{longtable}
 
$(\mathfrak{L}_{\lambda}, \cdot, [-,-])$ is a Poisson algebra if and only if
$(\gamma, \alpha_1, \ldots \alpha_n, \beta_1, \ldots, \beta_n)=0.$
\end{thrm}

\begin{proof}
We aim to describe the multiplication $\cdot.$
By Lemma \ref{glavlem}, 
for each  $\ell \in \mathbb{B}$, there is a related 
$\frac{1}{2}$-derivation $\varphi_{\ell}$ of $(\mathfrak{L}_{\lambda}, [-,-])$
such that $\varphi_{\ell_1}(\ell_2 )=\ell_1 \cdot \ell_2 = \varphi_{\ell_2}(\ell_1),$ 
with $\ell_1, \ell_2 \in \mathbb{B}.$
By Proposition \ref{dmatrix}, for each  $\ell \in \mathbb{B}$, we have

\begin{longtable}{rcl}
$\varphi_{\ell}(e_{-1})$ & $ =$ &$\gamma^{\ell} e_{-1} + \mu^{\ell} e_0 - \sum\limits_{k=1}^{n} 2\lambda_k \alpha^{\ell}_{k} e_k - \sum\limits_{k=1}^{n} 2 \lambda_k \beta^{\ell}_{k} \check{e}_k$,\\
$\varphi_{\ell}(e_0)$ & $ =$& $  \gamma^{\ell} e_0$,\\
$\varphi_{\ell}(e_j) $&$=$&$ \alpha^{\ell}_{j} e_0 + \gamma^{\ell} e_j, \ \ \ (j=1, \ldots, n)$,\\
$\varphi_{\ell}(\check{e}_j)$ & $ =$&$ \beta^{\ell}_{j} e_0 + \gamma^{\ell} \check{e}_{j}, \ \ \ (j=1, \ldots, n)$.
\end{longtable}
\noindent First, for all $i,j \in \{1, \ldots, n\},$ we have 
\begin{longtable}{rcccccl}

$\alpha^{e_{i}}_{j} e_0 + \gamma^{e_{i}} e_j $&$=$&$ \varphi_{e_{i}}(e_j) $&$=$&$ \varphi_{e_{j}}(e_i) $&$=$&$ \alpha^{e_{j}}_{i} e_0 + \gamma^{e_{j}} e_i$,\\
$\alpha^{\check{e}_i}_{j} e_0 + \gamma^{\check{e}_i} e_j$&$ =$&$ \varphi_{\check{e}_i}(e_j) $&$=$&$  \varphi_{e_{j}}(\check{e}_i)$&$ =$&$ \beta^{e_{j}}_{i} e_0 + \gamma^{e_{j}} \check{e}_{i}$,\\
$\beta^{\check{e}_i}_{j} e_0 + \gamma^{\check{e}_i} \check{e}_{j}$&$ =$&$ \varphi_{\check{e}_i}(\check{e}_j)$&$ = $&$\varphi_{\check{e}_j}(\check{e}_i) $&$ =$&$\beta^{\check{e}_j}_{i} e_0 + \gamma^{\check{e}_j} \check{e}_{i}.$
 
\end{longtable}

\noindent Hence 
$\beta^{\check{e}_j}_{i} = \beta^{\check{e}_i}_{j}$, 
$\alpha^{e_{i}}_{j}=\alpha^{e_{j}}_{i} $, 
$\alpha^{\check{e}_i}_{j} = \beta^{e_{j}}_{i}$ and 
$\gamma^{e_{j}} = \gamma^{\check{e}_j} =  0$.

\noindent Second, we have
$$
\gamma^{e_{-1}}  e_0 = \varphi_{e_{-1}}(e_0) =\varphi_{e_{0}}(e_{-1}) =  \gamma^{e_{0}} e_{-1}+\mu^{e_0}e_0
-2 \sum\limits_{k=1}^n \lambda_k (\alpha_k^{e_0} e_k+\beta_k^{e_0} \check{e}_k),$$
which implies $ \alpha^{e_0}_j=\beta^{e_0}_j= \gamma^{e_{0}} = 0$ and $\gamma^{e_{-1}}=\mu^{e_0}$.

Also,  for all $ j \in \{1, \ldots, n\},$ we have 
\begin{longtable}{rcccccl}
$\alpha^{e_{-1}}_{j} e_0 + \gamma^{e_{-1}} e_j$&$ = $&$ \varphi_{e_{-1}}(e_j) $&$ = $&$ \varphi_{e_{j}}(e_{-1}) $&$ = $&$    
\mu^{e_{j}} e_0 - 2\sum\limits_{k=1}^{n} \lambda_k (\alpha^{e_{j}}_{k} e_k+\beta^{e_{j}}_{k} \check{e}_k),$\\
$\beta^{e_{-1}}_{j} e_0 + \gamma^{e_{-1}} \check{e}_{j} $&$ = $&$ \varphi_{e_{-1}}(\check{e}_j)$&$ = $&$ \varphi_{\check{e}_j}(e_{-1}) $&$ = $&$  
 \mu^{\check{e}_j} e_0 - 2\sum\limits_{k=1}^{n}  \lambda_k( \alpha^{\check{e}_j}_{k} e_k +  \beta^{\check{e}_j}_{k} \check{e}_k),$
\end{longtable}
\noindent which implies
\begin{longtable}{llll}
$\gamma^{e_{-1}} =  -2\lambda_j \alpha^{e_j}_j,$ & 
$ \alpha^{e_j}_j=\beta^{\check{e}_j}_j ,$ &  
$\alpha^{\check{e}_j}_{k} = 0,$ & 
$\beta^{e_j}_k = 0,$ \\ 

$\alpha^{e_j}_k = 0 \ (j\neq k),$ & 
$\beta^{\check{e}_j}_{k} = 0  \ (j\neq k),$ & 
$\alpha^{e_{-1}}_j = \mu^{e_j},$ & 
$\beta^{e_{-1}}_j = \mu^{\check{e}_j}$.
\end{longtable}

\noindent It follows that, for all $j \in \{1, \ldots, n\}$,   the multiplication $\cdot$ can be written as
\begin{longtable}{lcl}
$e_{-1} \cdot e_{-1}$ & $ =$ &$ \gamma  e_{-1}+ \mu e_0 - 2\sum\limits_{k=1}^{n} \lambda_k (\alpha_{k} e_k +\beta_{k} \check{e}_k),$\\
$ e_{-1} \cdot e_{0}$&$ =$&$ \gamma e_0,$\\
$e_{-1} \cdot e_{j}$ & $ =$ & $ \alpha_{j} e_0 + \gamma e_{j}, $\\
$e_{-1} \cdot \check{e}_{j}$ & $ =$ & $ \beta_{j} e_0 + \gamma \check{e}_{j},$\\
$e_j \cdot e_j$ & $ =$ & $ \check{e}_j \cdot \check{e}_j = -\frac{\gamma}{2\lambda_j} e_0  .$
\end{longtable}

\noindent It is easy to see that $\cdot$ is a commutative associative multiplication.\end{proof}

\subsection*{Isomorphism problem for transposed Poisson algebras on generic oscillator Lie algebras}
An oscillator Lie algebra is called generic if 
$0< \lambda_1 < \ldots  < \lambda_n$ and $\lambda_i+\lambda_j \neq \lambda_k$ for all $1\leq i < j<k \leq n.$ 
It is easy to see that the algebra 
$(\mathfrak{L}_{ \{ \lambda_1, \lambda_2, \ldots, \lambda_n\}}, [-,-])$ is isomorphic to
$(\mathfrak{L}_{ \left\{ 1, \frac{\lambda_2}{\lambda_1}, \ldots, \frac{\lambda_n}{\lambda_1}\right\}}, [-,-]).$ 
Hence, we will consider only the case $\lambda_1=1.$

 To prove the main result of this Section, we should give the following useful Lemma.
 
 \begin{lemma}\label{autgen}
 Let $\phi$ be an   automorphism of $(\mathfrak{L}_{\lambda}, [-,-])$.
 Then $\phi$ has the following types:
 \begin{longtable}{lcl}
$ \phi(e_{-1})$&$=$&$  \pm e_{-1} + \nu e_0+\sum\limits_{i=1}^n \nu_i e_i+\sum\limits_{i=1}^n \check{\nu}_i \check{e}_i,$\\ 
$ \phi(e_0)$&$=$&$\pm \xi e_0, $\\ 
$ \phi(e_i)$&$=$&$  \frac{ \check{\nu}_i\check{\mu}_i \mp \nu_i \mu_i}{\lambda_i}e_0+  \mu_i e_i \mp \check{\mu}_i \check{e}_i,$\\ 
$ \phi(\check{e}_i)$&$=$&$   \frac{ -\check{\nu}_i\mu_i \mp \nu_i \check{\mu}_i}{\lambda_i}e_0+  \check{\mu}_i e_i \pm \mu_i \check{e}_i,$\\  
      \end{longtable}
\noindent  where $\xi= \mu_i^2 +\check{\mu}_i^2$ does not depend on $i$.    
 \end{lemma}
 \begin{proof}
 Put 
 \begin{center}
     $\phi(e_m)= \alpha^{m}_{-1} e_{-1}+\alpha^{m}_{0} e_{0}+\sum\limits_{i=1}^{n}(\alpha_i^m e_i+\beta_i^m \check{e}_i),$ where $m\in \{-1,0,1, \ldots, n\}.$
      \end{center}

\noindent Since $\Ann(\LL)$ and $[\LL,\LL]$ are invariant under any automorphism, we have that 
$$\alpha_{-1}^{0} =\alpha_{i}^{0}=\beta_{i}^{0}=0, \quad \alpha_{i}^{-1}=0, \quad 1 \leq i \leq n.$$

\noindent Then, for each $j \in \{1,\ldots, n\},$ we have
\begin{multline*}\lambda_j^2 \phi(e_j)=-[\phi(e_{-1}),[\phi(e_{-1}),\phi(e_{j})]]=\\
(\alpha^{-1}_{-1})^2\sum\limits_{i=1}^n\lambda_i^2(\alpha_i^j e_i+\beta_i^j \check{e}_j)
-\alpha^{-1}_{-1}\sum\limits_{i=1}^n(\alpha_{i}^{-1}\alpha_i^j+\beta_i^{-1}\beta_i^j)e_0.
\end{multline*}

\noindent Since $\lambda_i\neq \lambda_j,$ we have that 
$$\alpha_{-1}^{-1}=\pm1,  \quad \alpha^j_0 = -\frac{\alpha_{-1}^{-1}}{\lambda_j}(\alpha_{j}^{-1}\alpha_j^j+\beta_j^{-1}\beta_j^j), \quad \alpha_i^j = \beta^j_i =0, \quad 1 \leq i (i\neq j) \leq n.$$

\noindent Thus, we obtain $$\phi(e_j)=-\frac{\alpha_{-1}^{-1}}{\lambda_j}(\alpha_{j}^{-1}\alpha_j^j+\beta_j^{-1}\beta_j^j)e_0+\alpha_j^je_j+\beta^j_j\check{e}_j.$$

\noindent Similarly, considering 
$\lambda_j^2 \phi(\check{e}_j)=-[\phi(e_{-1}),[\phi(e_{-1}),\phi(\check{e}_{j})]],$ we have 
$$\phi(\check{e}_j)=-\frac{\alpha_{-1}^{-1}}{\lambda_j}(\alpha_{j}^{-1}\mu_j^j+\beta_j^{-1}\nu_j^j)e_0+\mu_j^je_j+\nu^j_j\check{e}_j.$$

\noindent Finally, using the equalities $$\phi[e_{i},\check{e}_{i}] = [\phi(e_{i}),\phi(\check{e}_{i})],\quad \phi[e_{-1},{e}_{i}] = [\phi(e_{-1}),\phi({e}_{i})], \quad \phi[e_{i},\check{e}_{i}] = [\phi(e_{i}),\phi(\check{e}_{i})],$$
we derive that $\phi$ has the form from our statement.
\end{proof}

We define the following set
\begin{center}
    $\mathbb C_{>0}= \{ a+ b {\rm  i}  \}_{a>0} \cup \{  b{\rm  i} \}_{b\geq0}.$ 
\end{center}

\begin{thrm}
Let $(\mathfrak{L}_{\lambda},  [-,-])$ be the generic oscillator Lie algebra. If $(\mathfrak{L}_{\lambda}, \cdot, [-,-])$ is a transposed  Poisson algebra structure on $(\mathfrak{L}_{\lambda}, [-,-])$, 
then  $(\mathfrak{L}_{\lambda}, \cdot)$ 
has the following multiplication.
\begin{enumerate}[(A)]
  \item  If  $(\mathfrak{L}_{\lambda}, \cdot)$ is non-nilpotent, then 
    \begin{enumerate}
    \item $\gamma\neq0;$
    \item 
$e_{-1} \cdot e_{-1}=\gamma  e_{-1}, \
e_{-1} \cdot e_{0}= \gamma e_0, \
e_{-1} \cdot e_{j}= \gamma e_{j}, \
e_{-1} \cdot \check{e}_{j}=  \gamma \check{e}_{j}, \
e_j \cdot e_j= \check{e}_j \cdot \check{e}_j = -\frac{\gamma}{2\lambda_j} e_0;$
\item Two transposed Poisson algebras $({\mathfrak L}_{ \{\lambda,\gamma\}},\cdot,[-,-])$ and
   $({\mathfrak L}_{ \{\nu,\delta\}},\cdot,[-,-])$ are isomorphic  if and only if $\nu=\lambda,$ $\delta=\pm\gamma;$
\item All these algebras are non-Poisson.
    \end{enumerate}
    
      \item If  $(\mathfrak{L}_{\lambda}, \cdot)$ is  nilpotent, then $\gamma=0$ and
  $(\mathfrak{L}_{\lambda}, \cdot)$ has the  following multiplication table:
\begin{center}
$e_{-1} \cdot e_{-1}=   \mu e_{0} -2\sum\limits_{j=1}^n \lambda_j (\alpha_j e_j+\beta_j \check{e}_{j}), \quad
e_{-1} \cdot e_{j}=  \alpha_j e_{0}, \quad
e_{-1} \cdot \check{e}_{j}=  \beta_j e_{0},$
\end{center}
where $(\mu,\alpha_1,\ldots,\alpha_n,\beta_1,\ldots,\beta_n)\neq0$ and
\begin{enumerate}
    \item for each $j$, either $\alpha_j=0$ and $\beta_j\in\mathbb C_{>0}$, or $(\alpha_j,\beta_j)=(1,\varepsilon_j{\rm i})$ with $\varepsilon_j\in\{1,-1\}$;
    \item if $(\alpha_j,\beta_j)=(1,\varepsilon_j{\rm i})$ for some $j$, then the first such $\varepsilon_j$ is equal to $1$;
    \item $\mu\in\{0,1\}$, and if $\mu=0$, then the first nonzero element of the sequence $\{\beta_j : \alpha_j=0\}$ (if it exists) is equal to $1$.
\end{enumerate}
The corresponding transposed Poisson algebras are pairwise non-isomorphic;
and they are non-Poisson if and only if $(\alpha_j,\beta_j)\neq(0,0)$ for some $j$.
\end{enumerate}

\end{thrm}

\begin{proof}
Let us consider the transposed Poisson structure  $(\mathfrak{L}_{\lambda}, \cdot , [-,-])$ 
on the generic oscillator Lie algebra $(\mathfrak{L}_{\lambda},  [-,-])$ 
and $\cdot$ is given as in Theorem \ref{theodot}.
Under the action of a positive automorphism of the Lie algebra $(\mathfrak{L}_{\lambda},  [-,-]),$
given in Lemma \ref{autgen}, we rewrite the multiplication table of $(\mathfrak{L}_{\lambda}, \cdot)$ by the following way:

\begin{longtable}{lcl}
$e_{-1} \cdot e_{-1} $&$=$&$ \gamma e_{-1}+\left( \frac{\mu+\gamma \nu}{\xi}+\sum\limits_{i=1}^{n}\frac{\gamma(\nu_i^2+\check{\nu}_i^2)}{2 \lambda_i \xi}\right)e_0+$\\
& & \multicolumn{1}{r}{$\sum\limits_{k=1}^n\left( \frac{\gamma( \mu_k\nu_k- \check{\mu}_k\check{\nu}_k)-2\lambda_k(\alpha_k  \mu_k-\beta_k  \check{\mu}_k) }{\xi}\right)e_k+
\sum\limits_{k=1}^n\left( \frac{\gamma( \check{\mu}_k\nu_k+ \mu_k\check{\nu}_k)-2\lambda_k(\alpha_k  \check{\mu}_k+\beta_k  \mu_k) }{\xi}\right)\check{e}_k$}\\
$e_{-1} \cdot e_{0} $&$=$&$\gamma e_0$\\

$e_{-1} \cdot e_{k} $&$=$&$\left( \frac{\gamma( \mu_k\nu_k- \check{\mu}_k\check{\nu}_k)-2\lambda_k(\alpha_k  \mu_k-\beta_k  \check{\mu}_k) }{2\lambda_k\xi}\right)e_0 +\gamma e_k$\\

$e_{-1} \cdot \check{e}_{k} $&$=$&$\left( \frac{\gamma( \check{\mu}_k\nu_k+ \mu_k\check{\nu}_k)-2\lambda_k(\alpha_k  \check{\mu}_k+\beta_k  \mu_k) }{2\lambda_k\xi}\right)e_0 +\gamma \check{e}_k$\\
$e_k \cdot e_k$ & $ =$ & $ \check{e}_k \cdot \check{e}_k = -\frac{\gamma}{2\lambda_k} e_0  .$
\end{longtable}

We consider two different situations.
\begin{enumerate} 
    \item Let $\gamma\neq 0$, then by taking 
\begin{center}    $\nu_k=\frac{2\alpha_k \lambda_k}{\gamma}, $
    $\check{\nu}_k=\frac{2\beta_k \lambda_k}{\gamma} $ and
    $\nu= -\sum\limits_{i=1}^{n}\frac{\gamma(\nu_i^2+\check{\nu}_i^2)}{2 \gamma \lambda_i  }-\frac{\mu}{\gamma},$
\end{center}
we have a multiplication table from part $(A.a)$ in our statement.

It is easy to see that, under the action of a positive automorphism of the Lie algebra $(\mathfrak{L}_{\lambda},  [-,-]),$
given in Lemma \ref{autgen}, all  multiplications $\cdot$ give  non-isomorphic transposed Poisson structures. 
Under the action of a negative automorphism given in Lemma \ref{autgen}, we have the isomorphism between 
structures corresponding to $\gamma$ and $-\gamma.$
Moreover, $(\mathfrak{L}_{\lambda},   \cdot )$ is non-nilpotent if and only if $\gamma \neq 0$.

    \item Let $\gamma= 0,$ then $(\mathfrak{L}_{\lambda},   \cdot )$ is a nilpotent algebra. A direct computation shows that an automorphism from Lemma \ref{autgen} with the upper (resp., lower) sign maps the structure with parameters $(\mu,\alpha_k,\beta_k)$ to the structure with parameters
  $$\mu'=\xi\mu, \quad (\alpha_k',\beta_k')=(\mu_k\alpha_k+\check{\mu}_k\beta_k,\ \mu_k\beta_k-\check{\mu}_k\alpha_k)$$
  (resp., $\mu'=-\xi\mu$, $(\alpha_k',\beta_k')=(\mu_k\alpha_k+\check{\mu}_k\beta_k,\ \check{\mu}_k\alpha_k-\mu_k\beta_k)$), where $\mu_k^2+\check{\mu}_k^2=\xi$ for all $k$.
  In particular, $\alpha_k'^2+\beta_k'^2=\xi(\alpha_k^2+\beta_k^2)$, and the conditions $(\alpha_k,\beta_k)=(0,0)$ and $\alpha_k^2+\beta_k^2=0$ are invariant.
  \begin{enumerate}
     \item If $\alpha_k^2+\beta_k^2\neq0,$ then choosing $(\mu_k,\check{\mu}_k)$ with $\mu_k\alpha_k+\check{\mu}_k\beta_k=0$ we obtain $\alpha_k'=0$; the remaining parameter $\beta_k'$ is defined by $\beta_k'^2=\xi(\alpha_k^2+\beta_k^2)$, i.e., up to the sign.
     \item If $(\alpha_k,\beta_k)\neq(0,0)$ and $\alpha_k^2+\beta_k^2=0$, then $\beta_k=\varepsilon_k{\rm i}\alpha_k$, $\varepsilon_k\in\{1,-1\}$; the vector $(\alpha_k,\beta_k)$ cannot be transformed to a vector with $\alpha_k'=0$, but, since $\alpha_k'=(\mu_k+\varepsilon_k{\rm i}\check{\mu}_k)\alpha_k$ and $\mu_k+\varepsilon_k{\rm i}\check{\mu}_k$ can be chosen arbitrary non-zero, we obtain $(\alpha_k',\beta_k')=(1,\varepsilon_k{\rm i})$. The automorphisms with the lower sign change all $\varepsilon_k$ simultaneously.
     \item Finally, if $\mu\neq0$, then choosing $\xi=\mu^{-1}$ we obtain $\mu'=1$, and the remaining parameters $\beta_k'$ (with $\alpha_k'=0$) are defined up to the sign; if $\mu=0$, then $\xi$ can be used to obtain $\beta_t'=1$ for the first $t$ with $\alpha_t'=0$ and $\beta_t'\neq0.$
  \end{enumerate}
  It follows that $(\mathfrak{L}_{\lambda}, \cdot)$ is isomorphic to exactly one structure from the part $(B)$ of our statement. Since the structure is Poisson if and only if $(\gamma,\alpha_1,\ldots,\alpha_n,\beta_1,\ldots,\beta_n)=0$ (Theorem \ref{theodot}), the last statement follows.
\end{enumerate}
\end{proof}

\section{Transposed Poisson structures on finite-dimensional solvable Lie algebras}\label{solv}

It is proven that any finite-dimensional solvable Lie algebra over an algebraically closed field of zero characteristic admits non-trivial $\frac12$-derivations \cite{klv22}. For the solvable Lie algebra $\mathfrak L$ with non-zero annihilator,  there is a non-trivial $\frac12$-derivation $\varphi$, such that  $\varphi(x) = [x, \omega_0],$
where   $\omega_0 \in \operatorname{Ann}_{[\mathfrak{L},\mathfrak{L}]}([\mathfrak{L},\mathfrak{L}])$ and $[x,\omega_0] =\lambda_x\omega_0$ for all $x \in \mathfrak{L},$ and $\lambda_x \neq 0$ for some $x.$ Thus, the dimension of the space of $\frac12$-derivations of any solvable Lie algebra is at least equal to two. 
In the following example, we give a solvable Lie algebra in which the dimension of the space of $\frac12$-derivations is equal to two.

\subsection{Transposed Poisson structures on solvable Lie algebras with a naturally graded filiform nilpotent radical}
All solvable Lie algebras with a naturally graded filiform nilpotent radical were found in \cite{sw05}.
Now we consider the  solvable Lie algebra $\mathfrak{s}_{n,2}$ with the following multiplication table:
\begin{longtable}{lcll lcll}
$[e_i, e_1] $&$=$&$ e_{i+1},$ \ &   $2 \leq i \leq n-1;$ &
$[e_1, x_1] $&$=$&$ e_1,$ \\
$[e_i, x_1] $&$=$&$ (i-2)e_i,$\  & $3 \leq i \leq n;$ &
$[e_i, x_2] $&$=$&$ e_i,$ \ &  $2 \leq i \leq n.$
\end{longtable}

Note that $\mathfrak{s}_{n,2}$ is a solvable Lie algebra, whose nilradical is the naturally graded filiform Lie algebra
$$\mathfrak{n}_{n,1} : [e_i, e_1]=e_{i+1}, \quad 2\leq i \leq n-1.$$

Moreover, $[\mathfrak{s}_{n,2},\mathfrak{s}_{n,2}] = \mathfrak{n}_{n,1}$ and $\operatorname{Ann}_{[\mathfrak{s}_{n,2},\mathfrak{s}_{n,2}]}([\mathfrak{s}_{n,2},\mathfrak{s}_{n,2}]) = \langle e_n\rangle.$ 
In the following proposition, we give the description of the space of $\frac 1 2$-derivations of the algebra $\mathfrak{s}_{n,2}.$

\begin{proposition}  
Any $\frac 1 2$-derivation $\varphi$ of the algebra $\mathfrak{s}_{n,2}$
has the form
$$\varphi(x_1) = \alpha x_1 + (n-2)\beta e_n, \quad 
\varphi(x_2) = \alpha x_2 + \beta e_n, \quad 
\varphi(e_k) = \alpha e_k, \quad 1 \leq k  \leq n.$$
\end{proposition}

\begin{thrm} There is only one non-trivial transposed Poisson algebra structure on the Lie algebra $\mathfrak{s}_{n,2},$ which is given by
\begin{equation*}
 x_1 \cdot x_1 = (n-2)^2 e_n, \quad x_2 \cdot x_1 =  (n-2) e_n, \quad x_2 \cdot x_2 = e_n.   
\end{equation*}
This structure is non-Poisson.

\end{thrm}

\begin{proof} 
Let $(\mathfrak{s}_{n,2}, \cdot, [-,-])$ be a transposed Poisson algebra structure defined on the Lie
algebra $\mathfrak{s}_{n,2}.$ Then for any element of 
$x\in \mathfrak{s}_{n,2},$ the operator of multiplication
$\varphi_x(y) = x \cdot y$ is a  $\frac 1 2$-derivation
of $(\mathfrak{s}_{n,2}, [-,-])$. 
Hence,  we have that

{\small  

\begin{longtable}{lcl}
$\varphi_{x_1}(x_1)$&$ = $&$\alpha_{n+1}x_1 + (n-2)\beta_{n+1}e_n,$\\
$\varphi_{x_1}(x_2)$&$ = $&$ \alpha_{n+1} x_2 + \beta_{n+1} e_n,$\\
$\varphi_{x_1}(e_k)$&$ =$&$ \alpha_{n+1} e_k,$\\
$\varphi_{x_2}(x_1) $&$=$&$ \alpha_{n+2}x_1 + (n-2)\beta_{n+2}e_n,$\\
$\varphi_{x_2}(x_2) $&$=$&$ \alpha_{n+2} x_2 + \beta_{n+2} e_n,$\\
$\varphi_{x_2}(e_k) $&$=$&$ \alpha_{n+2} e_k,$\\
$\varphi_{e_i}(x_1) $&$=$&$ \alpha_{i} x_1 + (n-2)\beta_ie_n,$\\
$\varphi_{e_i}(x_2) $&$=$&$ \alpha_i x_2 + \beta_i e_n,$\\
$\varphi_{e_i}(e_k) $&$=$&$ \alpha_i e_k.$
\end{longtable}}

\noindent Then from $\alpha_i e_j = \varphi_{e_i}(e_j) = e_i \cdot  e_j = e_j \cdot  e_i  =  \varphi_{e_j}(e_i) = \alpha_j e_i,$ we obtain that $\alpha_i =0$ for $1 \leq i \leq n.$
The last  implies $e_i \cdot e_j =0 $ for all $1 \leq i, j \leq n.$
Next from 
\begin{longtable}{llllllllllllll}
    $\alpha_{n+1} e_1 $&$=$&$ \varphi_{x_1}(e_1) $&$=$&$ e_1 \cdot  x_1$&$ =$&$ x_1 \cdot  e_1 $&$ =$&$ \varphi_{e_1}(x_1) $&$= $&$(n-2)\beta_1 e_n,$\\
 $\alpha_{n+2} e_1 $&$=$&$ \varphi_{x_2}(e_1)$&$=$&$ e_1 \cdot  x_2$&$ =$&$ x_2 \cdot  e_1 $&$ =$&$ \varphi_{e_1}(x_2) $&$=$&$ \beta_1 e_n,$
 \end{longtable}
\noindent we get that $\alpha_{n+1} =\alpha_{n+2} =0.$
After that, from 
\begin{longtable}{lllllllllllll}
    $0$&$=$&$\alpha_{n+2} e_i $&$=$&$ \varphi_{x_2}(e_i)$&$ = $&$e_i \cdot  x_2 $&$=$&$ x_2 \cdot  e_i  $&$=$&$ \varphi_{e_i}(x_2) $&$=$&$ \beta_i e_n,$    
\end{longtable} \noindent we obtain that $\beta_i =0$ for $1 \leq i \leq n$
and $e_i \cdot x_1=e_i\cdot x_2=0.$
Finally, considering 
\begin{longtable}{llllllllllllllllllllll}
    $(n-2)\beta_{n+2}e_n $&$=$&$ \varphi_{x_2}(x_1) $&$=$&$ x_2 \cdot  x_1 $&$=$&$ x_1 \cdot x_2  $&$=$&$ \varphi_{x_1}(x_2) $&$=$&$ \beta_{n+1} e_n,$
\end{longtable}
\noindent we derive $\beta_{n+1} =(n-2)\beta_{n+2}.$ Thus, we obtain the multiplications
\begin{longtable}{lllllllllll}
$x_1 \cdot x_1 $&$= $&$(n-2)^2 \gamma e_n,$ & 
$x_2 \cdot x_1 $&$= $&$ (n-2) \gamma e_n,$ & 
$x_2 \cdot x_2 $&$=$&$ \gamma e_n.$
\end{longtable}

\noindent Since we are interested only in non-trivial transposed Poisson algebra structures, we can suppose $\gamma \neq 0,$ and via automorphism 
\begin{longtable}{llllllllll}
$\phi(x_1) $&$= $&$x_1,$ &$\phi(x_2) = x_2,$ & $\phi(e_1) = e_1,$ &  
$\phi(e_i) = \gamma^{-1}{e_i},$ \ & $2 \leq i \leq n,$
\end{longtable}
\noindent we conclude that these transposed Poisson algebra structures are isomorphic to the case  $\gamma =1.$
\end{proof}

\subsection{Transposed Poisson structures on  solvable Lie algebras with Heisenberg nilpotent radical}
All solvable Lie algebras with Heisenberg nilpotent radical were found in  \cite{rw93}.
In this subsection, we consider $(3n+2)$-dimensional solvable Lie algebra $L_{n,n+1}$ with the following multiplication table: 
\begin{longtable}{lcllcllcll}
$[e_{n+i}, e_i]$&$=$&$e_{2n+1},$ &    $[e_i, x_{i}] $&$ =$&$ e_i,$  & $[e_{n+i}, x_{i}] $&$ =$&$ -e_{n+i},$  & $1 \leq i \leq n,$\\
& & & $[e_i, x_{n+1}] $&$=$&$ e_i,$  & $[e_{2n+1}, x_{n+1}] $&$= $&$e_{2n+1},$ & $1 \leq i \leq n.$
\end{longtable}

Note that the nilpotent radical of this algebra is the $(2n+1)$-dimensional 
Heisenberg algebra ${\rm H}_n$ with a basis $\{e_1, e_2, \dots, e_{2n+1}\}.$ 
Here $[L_{n,n+1},L_{n,n+1}] = {\rm H}_n$ and $\operatorname{Ann}_{{\rm H}_n}({\rm H}_n) = \langle e_{2n+1}\rangle.$
In the following, proposition we give the description of space of $\frac 1 2$-derivations of the algebra $L_{n,n+1}.$

\begin{proposition} \label{prop-Heisenberg} 
Any $\frac 1 2$-derivation $\varphi$ of the algebra $L_{n,n+1}$
has the form
\begin{longtable}{lcll}
$\varphi(e_k)$&$ =$&$ \alpha e_k, $&$ 1 \leq k  \leq 2n+1; $\\
$\varphi(x_k) $&$= $&$\alpha x_k, $&$ 1 \leq k  \leq n; $\\
$\varphi(x_{n+1}) $&$= $&$\alpha x_{n+1} + \beta e_{2n+1}.$&
\end{longtable}
\end{proposition}

\begin{thrm} 
There is only one non-trivial transposed Poisson algebra structure on the Lie algebra $(L_{n,n+1}, [-,-]),$ which is given by the multiplication 
 $$x_{n+1} \cdot x_{n+1} =  e_{2n+1}.$$
\noindent This structure is non-Poisson.
\end{thrm}

\begin{proof} 
Let $(L_{n,n+1}, \cdot, [-,-])$ be a transposed Poisson algebra structure defined on the Lie
algebra $L_{n,n+1}.$ 
Since the operator
$\varphi_x(y) = x \cdot y$ is a  $\frac 1 2$-derivation of $(L_{n,n+1}, [-,-])$ for any element of $x\in L_{n,n+1},$ we have that
{\small \begin{longtable}{lcll lcl}
$\varphi_{e_i}(e_k)$&$ = $&$\alpha_i e_k,$ & $1 \leq k  \leq 2n+1;$  \\
$\varphi_{e_i}(x_k)$&$ = $&$\alpha_i x_k,$ & $1 \leq k  \leq n+1;$ & $\varphi_{e_i}(x_{n+1}) $&$=$&$ \alpha_{i} x_{n+1} + \beta_ie_{2n+1},$ \\
$\varphi_{x_j}(e_k)$&$ =$&$ \alpha_{2n+1+j}e_k,$ & $1 \leq k  \leq 2n+1;$\\
$\varphi_{x_j}(x_k) $&$=$&$ \alpha_{2n+1+j} x_k,$ & $1 \leq k  \leq n;$ & 
$\varphi_{x_j}(x_{n+1}) $&$= $&$\alpha_{2n+1+j}x_{n+1} + \beta_{2n+1+j}e_{2n+1}.$
\end{longtable}}

\noindent Then from 
\begin{longtable}{lclclclclclc}
    $\alpha_i e_k $&$= $&$\varphi_{e_i}(e_k)$&$ =$&$ e_i \cdot  e_k $&$=$&$ e_k \cdot  e_i  $&$= $&$ \varphi_{e_k}(e_i) $&$=$&$ \alpha_k e_i,$    
\end{longtable} \noindent we obtain that $\alpha_i =0$ for $1 \leq i \leq 2n+1.$ Next from 
\begin{longtable}{lclclclclclc}
$\alpha_{2n+1+j} e_1 $&$=$&$ \varphi_{x_j}(e_1)$&$ =$&$ e_1 \cdot  x_j $&$=$&$ x_j \cdot  e_1  $&$=$&$ \varphi_{e_1}(x_j) $&$=$&$ 0,$
\end{longtable}
\noindent we have that $\alpha_{2n+1+j} =0$ for $1 \leq j \leq n.$ 
After then considering 
\begin{longtable}{lclclclclcl}
$\alpha_{3n+2} e_i$&$ =$&$ \varphi_{x_{n+1}}(e_i)$&$ =$&$ e_i \cdot  x_{n+1} $&$=$&$ x_{n+1} \cdot  e_i  $&$=$&$ \varphi_{e_i}(x_{n+1})$&$ =$&$ \beta_i e_{2n+1},$
\end{longtable}
\noindent we derive that $\alpha_{3n+2}=0$ and $\beta_i =0$ for $1 \leq i \leq 2n+1.$
Finally, considering 
\begin{longtable}{lclclclclcl}
$\beta_{2n+1+j}e_{2n+1} $&$=$&$ \varphi_{x_j}(x_{n+1}) $&$=$&$ x_j \cdot  x_{n+1} $&$ =$&$ x_{n+1} \cdot x_j  $&$= $&$\varphi_{x_{n+1}}(x_j) $&$= $&$0,$
\end{longtable} \noindent we get that $\beta_{2n+1+j} = 0$ for $1 \leq j \leq n$ and $x_i \cdot x_j = 0,$ $ 1 \leq i\leq n+1,$ $1 \leq j\leq n.$
Thus, we obtain the multiplication
$$ \quad x_{n+1} \cdot x_{n+1} = \gamma e_{2n+1}.$$

Since we are interested only in non-trivial transposed Poisson algebra structures, we have that $\gamma \neq 0$ and via  automorphism 
\begin{longtable}{lcll}
$ \phi(e_i)$&$ =$&$ e_i,$ & $1 \leq i \leq n;$\\
$\phi(e_j) $&$= $&$\gamma^{-1}{e_j} ,$ & $n+1 \leq j \leq 2n+1;$\\
$\phi(x_k)$&$=$&$ x_k,$ & $1 \leq k \leq n+1,$
\end{longtable}
\noindent we can suppose that $\gamma =1.$
\end{proof}

\subsection{Transposed Poisson structures on 
solvable Lie algebras with abelian nilpotent radical}

Now we consider solvable Lie algebras with abelian nilpotent radical and maximal complementary vector space
(these algebras were found in \cite{nw94}).
It is known that the maximal dimension of complementary space for solvable Lie algebras with $n$-dimensional abelian nilpotent radical is equal to $n.$ Moreover, up to isomorphism there exists only one such solvable Lie algebra with the following multiplications:
$$L_n: \left[e_i, x_i\right]=e_i, \quad 1 \leq i \leq n,$$
where $\{e_1,   \dots, e_n, x_1,  \dots, x_n\}$ is a basis of $L_n.$
In the following proposition, we give the description of the space of $\frac 1 2$-derivations of the algebra $L_{n}.$
\begin{proposition} Any $\frac 1 2$-derivation $\varphi$ of the algebra $L_{n}$
has the form
\begin{longtable}{lllllllll}
$\varphi(e_i) $&$=$&$ \alpha_i e_i,$  \  
& $\varphi(x_i) $&$= $&$ \alpha_i x_i + \beta_{i} e_i,$ \ $ 1 \leq i  \leq n.$
\end{longtable}
\end{proposition}

\begin{thrm}\label{thmL_n} Let $(L_n, \cdot, [-,-])$ be a transposed Poisson algebra structure defined on the Lie
algebra $L_{n}.$ Then the multiplication  of $(L_n, \cdot)$ has the following form:
$$e_i\cdot x_i=\mu_{i}e_{i}, \quad  x_i\cdot x_i = \mu_{i}x_i + \tau_{i}e_{i}, \quad  1 \leq i \leq n.$$
\end{thrm}

\begin{proof} $(L_n, \cdot, [-,-])$ be a transposed Poisson algebra structure defined on the Lie
algebra $L_{n}.$ Then, for any element of $x\in L_n,$ we have that the operator of multiplication
$\varphi_x(y) = x \cdot y$ is a  $\frac 1 2$-derivation. Hence,
for $ 1 \leq i,k  \leq n$ we derive
\begin{longtable}{lcll lcll}
$\varphi_{e_i}(e_k) $&$=$&$ \alpha_{i,k} e_k,$ &
$ \varphi_{e_i}(x_k) $&$= $&$\alpha_{i,k} x_k + \beta_{i,k}e_{k},$\\
$\varphi_{x_i}(e_k) $&$=$&$ \gamma_{i,k} e_k,$&
$ \varphi_{x_i}(x_k) $&$=$&$ \gamma_{i,k}x_k + \delta_{i,k}e_{k}.$
\end{longtable}


\noindent Thus, we obtain that
$$e_i\cdot e_k = 0, \quad 1 \leq i,k  \leq n; 
\quad e_j\cdot x_k=x_j\cdot x_k =0, \quad  1 \leq j\neq k \leq n.$$ 

\noindent Therefore, we derive that non-zero multiplications of the algebra $(L_n, \cdot)$ have the form:
$$e_i\cdot x_i=\beta_{i,i}e_{i}, \quad x_i\cdot x_i = \beta_{i,i}x_i + \delta_{i,i}e_{i}, \quad  1 \leq i  \leq n.$$
\end{proof}

 \begin{remark} From Theorem \ref{thmL_n}, we obtain that transposed Poisson algebra structure defined on the Lie
algebra $L_{n}$ is a direct sum of two-dimensional transposed Poisson algebras ${\rm TP}_i^{\mu_i, \tau_i},$ where
$${\rm TP}_i^{\mu_i, \tau_i}: \quad  
\left[e_i, x_i\right]=e_i, \quad e_i\cdot x_i=\mu_i e, \quad  x_i\cdot x_i = \mu_i x_i  + \tau_i e_i.$$
It is not difficult to see that any transposed Poisson algebra from the class ${\rm TP}_i^{\mu_i, \tau_i}$ is isomorphic one of the following non-isomorphic algebras
\begin{center}
    ${\rm TP}_1^\mu = {\rm TP}_i^{\mu, 0}$ 
    or ${\rm TP}_2={\rm TP}_i^{0, 1}.$
\end{center}

\noindent Therefore, we obtain that transposed Poisson algebra structure defined on the Lie algebra $L_{n}$ is a direct sum of some copies of algebras 
${\rm TP}_1^{\mu}$ and ${\rm TP}_2.$ 
 \end{remark}
 
\section{Transposed Poisson structure on semidirect products of ${\mathfrak sl}_2$ and irreducible modules}

In this section, we consider transposed Poisson structure for some finite-dimensional Lie algebras which is a semi-direct sum of semisimple and solvable algebras, i.e., $\mathfrak{L} = \mathfrak{s} \ltimes \mathfrak{r}$. More precisely, we consider Lie algebras with three-dimensional simple parts and the solvable radical is an irreducible representation.  
It is known that if the  simple algebra $\mathfrak{sl}_2$ has an $(m+1)$-dimensional irreducible representation  $\mathfrak{r},$ where $m \geq 2$ (see, \cite{jac62}), 
then we have the algebra $\mathfrak L^m = \mathfrak{sl}_2 \ltimes \mathfrak r$ with  the following multiplication: 
{\small 
\begin{longtable}{lcllcllcl} 
$[e,f]$&$=$&$h,$ & $[h,e]$&$=$&$2e,$ & $[f,h]$&$=$&$2f,$\\  
$[x_k,h]$&$=$&$(2k-m)x_k,$ & \multicolumn{6}{l}{$0 \leq k \leq m ,$}\\
$[x_k,f]$&$=$&$x_{k+1},$  & \multicolumn{6}{l}{$0 \leq k \leq m-1,$} \\
 $[x_k,e]$&$=$&$k(m+1-k)x_{k-1},$ & \multicolumn{6}{l}{$1 \leq k \leq m.$}
 \end{longtable}}

\begin{proposition} \label{prop29}  If  $m> 2$, then any $\frac 1 2$-derivation of $\mathfrak L^m$ is trivial.
 \end{proposition}

\begin{proof} Let $\varphi$ be a $\frac 1 2$-derivation of $\mathfrak L^m.$ Put
\begin{longtable}{llllll}
$\varphi(e)$&$ =$&$ \alpha^1_1 e + \alpha^1_2 f + \alpha^1_3 h + \sum\limits_{k=0}^m\beta^1_{k}x_k,$ & 
$\varphi(f) $&$=$&$ \alpha^2_1 e + \alpha^2_2 f + \alpha^2_3 h + \sum\limits_{k=0}^m\beta^2_{k}x_k.$
\end{longtable}

\noindent Then from $\varphi(h) = \varphi([e,f]) = \frac 1 2 \big([\varphi(e),f] + [e, \varphi(f)]\big),$ we have that
{\small \begin{longtable}{lll}
    $\varphi(h)$&$ = $&$-\alpha^2_3 e - \alpha^1_3 f + \frac{\alpha^1_1+\alpha^2_2} 2 h - \frac m 2 \beta^2_{1}x_0   + \frac 1 2 \sum\limits_{k=1}^{m-1}\Big(\beta^1_{k-1} - (k+1)(m-k)\beta^2_{k+1} \Big) x_{k} + \frac 1 2 \beta^1_{m-1}x_{m}.$
    \end{longtable}}

\noindent Consider
{\small
\begin{longtable}{lcl}
$\varphi([h,e])$&$ = $&$\frac 1 2 \big([\varphi(h), e] + [h, \varphi(e)]\big)=$\\
&$=$& 
$\frac 1 2 \Big([-\alpha^2_3 e - \alpha^1_3 f + \frac{\alpha^1_1+\alpha^2_2} 2 h - \frac m 2 \beta^2_{1}x_0  +  \sum\limits_{k=1}^{m-1}\frac{\beta^1_{k-1} - (k+1)(m-k)\beta^2_{k+1}}{2} x_{k} + \frac{ \beta^1_{m-1}}{2}x_{m}, e] +$\\

\multicolumn{3}{r}{$ +[h, \alpha^1_1 e + \alpha^1_2 f + \alpha^1_3 h + \sum\limits_{k=0}^m\beta^1_{k}x_k]\Big)$}\\

&$=$&$\frac{\alpha^1_3} 2 h + \frac{\alpha^1_1+\alpha^2_2} 2 e   + \sum\limits_{k=1}^{m-1}\frac{(\beta^1_{k-1} - (k+1)(m-k)\beta^2_{k+1})k(m+1-k)}{4} x_{k-1}+$\\

\multicolumn{3}{r}{$+\frac{m \beta^1_{m-1}}{4} x_{m-1} + \alpha^1_1 e - \alpha^1_2 f +    \sum\limits_{k=0}^m\frac{\beta^1_{k}(m-2k)}{2} x_k$}\\

&$=$&$\frac{3\alpha^1_1+\alpha^2_2} 2 e - \alpha^1_2 f + \frac{\alpha^1_3} 2 h +  \sum\limits_{k=0}^{m-2}\frac{(\beta^1_{k}- (k+2)(m-k-1)\beta^2_{k+2} )(k+1)(m-k)}{4} x_{k}+$\\
\multicolumn{3}{r}{$+\frac {m \beta^1_{m-1}}{4}x_{m-1} +  \sum\limits_{k=0}^m\frac{\beta^1_{k}(m-2k)}{2} x_k.$}
\end{longtable}
}

\noindent On the other hand, 
\begin{longtable}{lcllcllcl}
    $\varphi([h,e]) $&$ = $&$2 \varphi(e) $&$=$&$  2\alpha^1_1 e + 2\alpha^1_2 f + 2\alpha^1_3 h + 2\sum\limits_{k=0}^m\beta^1_{k}x_k.$
\end{longtable}
\noindent Comparing coefficients at the basis elements, we have
$$\alpha^1_1= \alpha^2_2, \quad \alpha^1_2=0, \quad \alpha^1_3=0, \quad \beta_{1, m}=0, \quad \beta^1_{m-1}=0,$$
and (for  $0 \leq k \leq m-2$):
\begin{equation}\label{eq8}\big((k+1)(m-k) +2m -4k -8\big) \beta_{1,k}= (k+2)(k+1)(m-k)(m-k-1)\beta^2_{k+2}.\end{equation}

\noindent Let us now consider 
{\small \begin{longtable}{lcl}
$\varphi([f,h]) $&$= $&$
\frac 1 2 \big([\varphi(f), h] + [f, \varphi(h)]\big)$\\
&$=$&$\frac 1 2\big[\alpha^2_1 e + \alpha^1_1 f + \alpha^2_3 h +\sum\limits_{k=0}^m\beta_{2,k}x_k, h\big] +$\\

&&$+\frac 1 2 \big[f, -\alpha^2_3 e + \alpha^1_1 h - \frac m 2 \beta^2_{1}x_0  +   
\sum\limits_{k=1}^{m-1}\frac{\beta^1_{k-1} - (k+1)(m-k)\beta^2_{k+1} }{2} x_{k} + \frac{\beta^1_{m-1}}{2}x_{m}\big]$ \\

&$=$&$-\alpha^2_1 e + \alpha^1_1 f +   
\sum\limits_{k=0}^m\frac{(2k-m)\beta_{2,k}}{2}x_k+ \frac{\alpha^2_3} 2 h + \alpha^1_1 f   + \frac{m \beta^2_{1}}{4}x_1  -   \sum\limits_{k=2}^{m}\frac{\beta_{1,k-2} - k(m-k+1)\beta^2_{k}}{4} x_{k}.$
\end{longtable}}

\noindent On the other hand,
\begin{longtable}{lclclclclc}
$\varphi([f,h])$&$ =$&$ 2 \varphi(f)$&$ = $&$ 2\alpha^2_1 e + 2\alpha^1_1 f + 2\alpha^2_3 h + 2\sum\limits_{k=0}^m\beta^2_{k}x_k.$
\end{longtable}

\noindent Comparing coefficients at the basis elements, we obtain
$$\alpha^2_1=0, \quad \alpha^2_3=0, \quad \beta_{2, 0}=0, \quad \beta_{2, 1}=0,$$ 
and
\begin{equation}\label{eq9}\beta_{1, k-2} = \big((k-2)m - k^2+5k-8\big)\beta_{2, k}, \quad 2 \leq k \leq m.\end{equation}

Then, from \eqref{eq8} and \eqref{eq9}, we have $(m-2)(m+4)\beta_{2, k} =0$ for $2 \leq k \leq m.$
Since $m\neq 2,$ we obtain $\beta_{2, k} =0,$ which implies $\beta_{1, k-2} =0$ for any  $2 \leq k \leq m.$
Thus, we get 
\begin{longtable}{lcllcllcl}
$\varphi(e) $&$=$&$\alpha^1_1 e,$&
$\varphi(f) $&$=$&$ \alpha^1_1 f,$&
$\varphi(h) $&$=$&$\alpha^1_1 h.$
\end{longtable}

\noindent Now put
\begin{longtable}{lcl}
$\varphi(x_0) $&$=$&$ \alpha^4_1 e + \alpha^4_2 f + \alpha^4_3 h + \sum\limits_{k=0}^m\beta^4_k x_k.$
\end{longtable}

\noindent From 
\begin{longtable}{lcl}
$0$&$=$&$2\varphi([x_0, e]) = [\varphi(x_0), e] + [x_0, \varphi(e)]  $\\
&$=$&$ [\alpha^4_1 e + \alpha^4_2 f + \alpha^4_3 h + \sum\limits_{k=0}^m\beta^4_k x_k, e] + [x_0, \alpha^1_1 e ]$\\
&$=$&$-\alpha^4_2 h + 2\alpha^4_3 e +
 \sum\limits_{k=1}^m k(m+1-k) \beta^4_k x_{k-1},$
\end{longtable}
\noindent we have that $\alpha^4_2 = \alpha^4_3  = \beta^4_k = 0,$ $1 \leq k \leq m.$
Considering the relation 
\begin{longtable}{lclclcl}
$2m\varphi(x_0) $&$ =$&$ 2\varphi([h, x_0])  =  [\varphi(h), x_0] + [h, \varphi(x_0)] $\\
& $=$ &$[\alpha^1_1 h, x_0] + [h, \alpha^4_1 e + \beta^4_0 x_0]= 2\alpha^4_1 e + m(\alpha^1_1 + \beta^4_0) x_1,$
\end{longtable}
\noindent we have that $\alpha^4_1 = 0$ and $\beta^4_0 = \alpha^1_1.$ Thus, $\varphi(x_0) = \alpha^1_1x_0.$ 
Then from $\varphi(x_{k}) = \varphi([x_{k-1}, f]),$ we obtain that
$\varphi(x_k) = \alpha^1_1x_k$ for $1\leq k\leq m.$ 
Hence, any $\frac 1 2$-derivation of $\mathfrak L^m$ is trivial.
\end{proof}

Surprisingly, in the case of $m=2$, we obtain a different result.
\begin{proposition} \label{prop30}
Any $\frac 1 2$-derivation $\varphi$ of $\mathfrak L^2$  
has the following form
\begin{longtable}{lcllcllcl}
$\varphi(e) $&$= $&$\alpha e -2\beta x_0,$ & 
$\varphi(f) $&$= $&$ \alpha f + \beta x_2,$ & 
$\varphi(h) $&$= $&$ \alpha h -2 \beta x_1,$ \\
$\varphi(x_0) $&$= $&$\alpha x_0,$ & $\varphi(x_1) $&$ = $&$ \alpha x_1,$ & $\varphi(x_2) $&$= $&$\alpha x_2.$
\end{longtable}
 \end{proposition}

 From the description of the $\frac 1 2$-derivation of the algebra $\mathfrak{L}^m,$ we have the following corollary. 
\begin{corollary} There are no non-trivial transposed Poisson algebra structures defined on the algebra $\mathfrak L^m.$  \end{corollary}
\begin{proof} Since in the case of $m\neq 2,$  the algebra $\mathfrak L^m$ has only trivial $\frac 1 2$-derivation,  we obtain that there are no non-trivial transposed Poisson algebra structures for $m\neq 2.$ 

Thus, it is sufficient to prove the Corollary for $m=2.$
Then we have the multiplication of $\mathfrak L^2$ has the form 
\begin{longtable}{lcllcllcl} 
$[e,f] $&$=$&$h,$ & $[h,e]$&$=$&$2e,$ & $[f,h]$&$=$&$2f,$ \\
$[x_0,h]$&$=$&$-2x_0,$ & $[x_0,f]$&$=$&$x_1,$ & $[x_1,e]$&$=$&$2x_0,$ \\
$[x_2,h]$&$=$&$2x_2,$ & $[x_1,f]$&$=$&$x_2,$ & $ [x_2,e]$&$=$&$2x_1.$ \\
\end{longtable}

By Lemma \ref{glavlem},  we have that for any element   
$\ell \in \{e, f, h, x_0, x_1, x_2\}$, there is a related 
$\frac{1}{2}$-derivation $\varphi_{\ell}$ of $\mathfrak L^2$.
Thus, by Proposition \ref{prop30}, we derive that
$$\begin{array}{lll} \varphi_{\ell}(e) = \alpha_{\ell} e -2\beta_{\ell} x_0, & \varphi_{\ell}(f) = \alpha_{\ell} f + \beta_{\ell} x_2, & \varphi_{\ell}(h) = \alpha_{\ell} h -2 \beta_{\ell} x_1, \\[1mm]
\varphi_{\ell}(x_0) = \alpha_{\ell} x_0, & \varphi_{\ell}(x_1) = \alpha_{\ell} x_1, & \varphi_{\ell}(x_2) = \alpha_{\ell} x_2.\end{array}$$
 
\noindent Then from 
\begin{longtable}{lcllcllcllcl}
$\alpha_{x_i} x_j $&$= $&$\varphi_{x_i}(x_j)$&$ =$&$ x_i \cdot x_j $&$= $&$\varphi_{x_j}(x_i) = \alpha_{x_j} x_i,$
\end{longtable}
\noindent we have that $\alpha_{x_0}=\alpha_{x_1} = \alpha_{x_2} =0.$  Now, considering 
\begin{longtable}{lclclclclclclcl}
$\alpha_{e} x_1 $&$=$&$ \varphi_{e}(x_1) $&$=$ &$e \cdot x_1 $&$= $&$\varphi_{x_1}(e) $&$=$&$ -2\beta_{x_1} x_0,$\\ 
$\alpha_{f} x_0 $&$=$&$ \varphi_{f}(x_0) $&$=$&$ f \cdot x_0 $&$=$&$ \varphi_{x_0}(f) $&$=$&$ \beta_{x_0} x_2,$\\ 
$\alpha_{h} x_2 $&$=$&$ \varphi_{h}(x_2) $&$=$&$ h \cdot x_2 $&$=$&$ \varphi_{x_2}(h) $&$=$&$ -2\beta_{x_2} x_1,$
\end{longtable}
\noindent we obtain that $\alpha_{e}=\alpha_{f} = \alpha_{h} = \beta_{x_0} = \beta_{x_1} = \beta_{x_2}=0.$
 Finally, from
\begin{longtable}{lclclclclclclcl}
$\beta_{e} x_2$ &$=$& $\varphi_{e}(f)$ &$=$ &$e \cdot f$ &$=$ &$\varphi_{f}(e) $&$=$&$ -2\beta_{f} x_0,$\\
$-2\beta_{e} x_1 $&$=$&$ \varphi_{e}(h) $&$= $&$e \cdot h $&$= $&$\varphi_{h}(e) $&$=$&$ -2\beta_{h} x_0,$
\end{longtable}
\noindent we get that $\beta_{e} = \beta_{f} = \beta_{h}=0.$ Thus, we conclude that there are no non-trivial transposed Poisson algebra structures on $\mathfrak L^2$.
\end{proof}

\begin{remark} It should be noted that the algebra $\mathfrak L^2$ is an algebra that admits non-trivial $\frac{1}{2}$-derivations, but
does not admit non-trivial transposed Poisson algebra structures.
Moreover, the algebra $\mathfrak L^2$ is a perfect algebra 
(i.e., $[\mathfrak L^2,\mathfrak L^2] =\mathfrak L^2$) 
and  admits non-trivial $\frac{1}{2}$-derivations.
The existence of algebras with this property gives a particular answer to \cite[Questions 2 and 4]{bfk22}.
    
\end{remark}

\appendix
\section*{Appendix A. Differences from the published version}
\addcontentsline{toc}{section}{Appendix A. Differences from the published version}

\medskip

\noindent\textbf{Theorem 11:} the statement has been changed and the proof corrected.

\end{document}